\documentclass{article}
\usepackage[a4paper]{geometry}

\usepackage{graphicx}
\usepackage{amsmath}
\usepackage{amssymb}
\usepackage{amsthm}
\usepackage{xcolor}
\usepackage{xspace}
\usepackage{array}
\usepackage{mathtools}
\mathtoolsset{centercolon}
\usepackage[
    colorlinks=true,
    citecolor = cyan,
    linkcolor=cyan,
]{hyperref}
\usepackage{cleveref}
\usepackage{autonum}
\usepackage{subcaption}
\usepackage{tikz}
\usetikzlibrary{positioning}
\usetikzlibrary{arrows}
\newcommand{\RR}{\mathbb{R}}

\newcommand{\CC}{\mathbb{C}}

\newcommand{\KK}{\mathbb{K}}
\newcommand{\rd}{\mathrm{d}}

\newcommand{\hsp}{\hspace*{1ex}}

\newcommand{\E}{E}
\newcommand{\Skew}{D} 
\newcommand{\VD}[2]{\nabla #1 \paren*{#2}}
 
\newcommand{\rU}{\mathrm{U}}
\newcommand{\rL}{\mathrm{L}}
\newcommand{\rX}{\mathrm{X}}
\newcommand{\cL}{\mathcal{L}}
\newcommand{\tE}{\tilde{E}}
\newcommand{\bu}{\bar{u}}

\newcommand{\Dt}{{\Delta t}}

\DeclarePairedDelimiter\paren{\lparen}{\rparen}
\DeclarePairedDelimiterX{\inpr}[2]{\langle}{\rangle}{{#1},{#2}}
\DeclarePairedDelimiterX{\L2ip}[2]{\langle}{\rangle_{L^2}}{{#1},{#2}}
\DeclarePairedDelimiterX{\setI}[2]{\{}{\}}{\,{#1}\ \delimsize| \ {#2}\,}

\DeclarePairedDelimiter{\abs}{|}{|}

\newtheorem{remark}{Remark}
\newtheorem{lemma}{Lemma}
\crefname{theorem}{Theorem}{Theorems}
\crefname{lemma}{Lemma}{Lemmas}
\crefname{proposition}{Proposition}{Propositions}
\crefname{corollary}{Corollary}{Corollaries}
\crefname{definition}{Definition}{Definitions}
\crefname{example}{Example}{Examples}
\crefname{remark}{Remark}{Remarks}
\crefname{figure}{Figure}{Figures}
\crefname{table}{Table}{Tables}

\newtheorem{assumption}{Assumption}
\crefname{assumption}{Assumption}{Assumptions}
\newtheorem{scheme}{Scheme}
\crefname{scheme}{Scheme}{Schemes}

\crefname{part}{\S}{\S\S}
\crefname{chapter}{\S}{\S\S}
\crefname{section}{Section}{Sections} 
\crefname{subsection}{\S}{\S\S}

\title{Scalar auxiliary variable approach \\ for conservative/dissipative partial differential equations \\ with unbounded energy}

\author{
    Tomoya Kemmochi%
    \thanks{
        Corresponding author. 
        Department of Applied Physics,
        Graduate School of Engineering,
        Nagoya University,
        Furo-cho, Chikusa-ku, Nagoya, Aichi, 464-8603, Japan. E-mail:
        \texttt{kemmochi@na.nuap.nagoya-u.ac.jp}
    }~
    and
    Shun Sato%
    \thanks{ 
        Department of Mathematical Informatics,
        Graduate School of Information Science and Technology,
        The University of Tokyo,
        Hongo 7-3-1, Bunkyo-ku, Tokyo, 113-0033, Japan.
    }
}

\date{}

\begin{document}
\maketitle

\begin{abstract}
In this paper, we present a novel investigation of the so-called SAV approach, which is a framework to construct linearly implicit geometric numerical integrators for partial differential equations with variational structure.
SAV approach was originally proposed for the gradient flows that have lower-bounded nonlinear potentials such as the Allen-Cahn and Cahn-Hilliard equations, and this assumption on the energy was essential.
In this paper, we propose a novel approach to address 
gradient flows with unbounded energy such as the KdV equation 
by a decomposition of energy functionals. 
Further, we will show that the equation of the SAV approach, which is a system of equations with scalar auxiliary variables, is expressed as another gradient system that inherits the variational structure of the original system.
This expression allows us to construct novel higher-order integrators by a certain class of Runge-Kutta methods.
We will propose second and fourth order schemes for conservative systems in our framework and present several numerical examples.

\end{abstract}

\section{Introduction}
\label{sec:intro}

In this paper, we consider geometric numerical integration of Partial Differential Equations (PDEs) of the form
\begin{equation}\label{eq:variationalPDEs}
    \dot{u} = \Skew \VD{\E}{u},
\end{equation}
where $ \Skew $ is a skew-adjoint or negative semidefinite operator defined over an appropriate Hilbert space $V$,
$ \nabla E $ is the Fr\'echet derivative of a functional $ \ E $, 
and $\dot{u}$ denotes the temporal derivative. (The details of the setting will be described later.) 
A class of PDEs of the form~\eqref{eq:variationalPDEs} involves many physically important equations, e.g., 
the Korteveg--de Vries equation, the Camassa--Holm equation~\cite{CH1993}, the Cahn--Hilliard equation~\cite{CH1958} and the Swift--Hohenberg equation~\cite{SH1977}. 
This paper is  devoted to novel investigation of the SAV approach for these equations, which is a framework to construct linearly implicit geometric numerical integrators recently proposed in~\cite{SXY2019JCP}.

A prominent property of the above PDEs is a conservation/dissipation law: 
\begin{equation}
    \frac{\rd}{\rd t} \E (u(t)) = \inpr*{ \VD{\E}{u} }{ \dot{u} } = \inpr*{ \VD{\E}{u} }{ \Skew \VD{\E}{u} }  \begin{cases} = 0 \quad (D\text{: skew-adjoint}),\\ \le 0 \quad (D\text{: negative semidefinite}),\end{cases}
\end{equation}
where $\inpr{\cdot}{\cdot}$ denotes the inner-product of $V$.
Therefore, for these PDEs, countless numerical schemes inheriting the conservation/dissipation law have been considered in the literature. 

In particular, there are some unified approaches to construct such schemes. 
The discrete gradient method~\cite{G1996,MQR1998,MQR1999} was originally devised for conservative/dissipative ODEs, which is now also used for PDEs in the form~\eqref{eq:variationalPDEs}~\cite{CGMMOOQ2012}. 
On the other hand, Furihata~\cite{F1999} independently proposed the Discrete Variational Derivative Method (DVDM) (see also~\cite{FM1998}) for the variational PDEs, which is now recognized as a combination of the discrete gradient method and an appropriate spatial discretization. 
The schemes based on the discrete gradient method (or DVDM) are often superior to general-purpose methods, in particular for a long run or numerically tough problems. 
Thus, they have been applied to many PDEs (see \cite{FM2011} and references therein).

However, at the price of their superiority, these schemes are unavoidably fully implicit and thus computationally expensive. 
Therefore, several remedies have been discussed in the literature. 
Some researchers have constructed structure-preserving and linearly implicit schemes for each specific PDE, for example, Besse~\cite{B2004} and Zhang, P\'{e}rez-Garc\'{\i}a, and V\'{a}zquez~\cite{ZPV1995} devised such schemes for the nonlinear Schr\"odinger equation. 
Moreover, for polynomial energy function, 
Matsuo and Furihata~\cite{MF2001} proposed a multistep linearly implicit version of the DVDM (see also Dahlby and Owren~\cite{DO2011}).

Most recently, for energy functions bounded below, Yang and Han~\cite{YH2017JCP} proposed an Invariant Energy Quadratization (IEQ) approach (see also~\cite{HBYT2017SC,Y2016JCP,ZWY2017IJNME}). 
A typical target of the IEQ approach is the Allen--Cahn equation~\cite{AC1979}
\begin{align}\label{eq:AC}
    u_t &= - \VD{\E}{u} = \Delta u - F'(u), &
    \E (u) &= \int_{\Omega} \paren*{ \frac{1}{2} \abs{\nabla u }^2 + F (u) } \rd x,
\end{align}
where $ F (u) = \frac{1}{4} \paren*{ u^2 - 1 }^2 $ is a potential function. 
Since $ F $ is non-negative, we can introduce an auxiliary function $ q = \sqrt{ F(u) + a } $ ($a$: a positive constant), and rewrite the Allen--Cahn equation~\eqref{eq:AC}. 
The reformulated one has a dissipation law of the modified energy 
$ \int_{\Omega} \paren*{ \frac{1}{2} \abs*{ \nabla u }^2 + q^2 } \rd x $. 
Since the modified energy is quadratic, we can easily construct dissipative schemes (see, e.g.,~\cite{GZW2019CPC,ZQS2020AML}). 
For example, we can use the implicit midpoint rule and its extensions, Gauss methods (see, e.g., \cite[Chapter~IV]{HLW2010}). 
Moreover, it is possible to construct linearly implicit schemes (see, e.g.,~\cite{MR4116799,HBYT2017SC,Y2016JCP,YH2017JCP,ZWY2017IJNME}). 

Though the IEQ approach is successful, it is still a bit expensive due to the presence of the auxiliary function. 
To overcome it, Shen, Xu, and Yang~\cite{SXY2019JCP} proposed a Scalar Auxiliary Variable (SAV) approach (see also \cite{SXY2019SIREV}). 
There, instead of the auxiliary function in the IEQ approach, a scalar auxiliary variable $ r = \sqrt{ \E_1 (u) + a } $ is introduced, where $ \E_1 (u) = \int_{\Omega} F(u) \rd x $. Then the Allen--Cahn equation~\eqref{eq:AC} can be rewritten in the form 
\begin{equation}\label{eq:AC_SAV}
    \begin{cases} 
    {\displaystyle u_t = \Delta u - \frac{ r }{\sqrt{\E_1 (u)+a}} F'(u),}\\[10pt]
    {\displaystyle \dot{r} = \frac{1}{2\sqrt{\E_1(u)+a}} \inpr*{ F'(u) }{ u_t} .}
    \end{cases}
\end{equation}
Again, since the modified energy $ \int_{\Omega} \frac{1}{2} \abs*{ \nabla u }^2 \rd x + r^2 $ is quadratic, we can construct dissipative schemes in several ways. 
The SAV approach often provides us with quite efficient numerical schemes: 
in addition to being the smaller system than the IEQ approach, in each step, the principal part is to solve two linear equations having the same $ d \times d $ constant coefficient matrix. 

Due to its attractive properties, the SAV approach has been intensively studied in these years. 
It has been applied to many PDEs, for example, the two-dimensional sine-Gordon equation~\cite{CJWS2019JCP}, the fractional nonlinear Schr\"odinger equation~\cite{FCW2019arx}, the Camassa--Holm equation~\cite{MR4083210}, and the imaginary time gradient flow~\cite{QJ2019JCP}. 
Shen and Xu~\cite{SX2018SINUM} and Li, Shen and Rui~\cite{LSR2019MC} conducted a convergence analysis of SAV schemes. 
Akrivis, Li, and Li~\cite{ALL2019SISC} devised and analyzed linear and high order SAV schemes. 
Moreover, the SAV approach is extended in several ways, for example, 
multiple SAV~\cite{CS2018SISC} and generalized Positive Auxiliary Variable (gPAV)~\cite{YD2020JCP}. 

Nevertheless, at least, there remain two issues on the SAV approach. 
First, the essential assumption ``the energy functional is bounded below'' is restrictive. 
In particular, when we deal with conservative PDEs, 
we often encounter an unbounded energy functional: for example, the KdV equation
\begin{align}\label{eq:KdV}
u_t &= \partial_x \VD{\E}{u},&
\E (u) &= \int_{\Omega} \paren*{ \frac{1}{2} \paren*{ u_x }^2 - u^3 } \rd x.
\end{align}
Second, the scalar auxiliary variable $r$ in the modified equation~\eqref{eq:AC_SAV} seems to be introduced in an unnatural way. 
As a consequence of this unclear derivation, the construction of resulting structure-preserving schemes is ad hoc.

There are some attempts to resolve the first problem~\cite{L2019arx1,L2019arx2,LL2019AML,MR4097160,LL2019arx2,bo2020arbitrary,cheng2020generalized}.
The strategy of these studies is to change the definition of the scalar auxiliary variable.
For example, in some of them, the authors replace the square-root function of the definition of $r$, such as the exponential function. 
Then, the auxiliary variable is always well-defined and thus the assumption for the energy functional can be removed.
Although these studies successfully overcome the first problem, the second issue, namely the natural derivation of the SAV schemes, is still remained.

The aim of this paper is twofold.
First we propose a novel approach to deal with unbounded energy functions by a decomposition of the energy functional.
Second, we present a new interpretation of the SAV approach using gradient system expression.
Furthermore, combining the above two results, we propose several numerical schemes in the framework of Crank-Nicolson method and the Runge-Kutta methods.

In order to address general energy functionals, we assume that the energy $E$ is expressed as
\begin{equation}
    E(u) = \frac{1}{2} \inpr{u}{Lu} + E_\rL(u) - E_\rU(u)
\end{equation}
for some linear operator $L$ and lower-bounded functions $E_\rL$ and $E_\rU$ (see~\cref{assumption}).
We then introduce two auxiliary variables $r_\rX$ ($\rX = \rL$ or $\rU$) by
\begin{equation}
    r_\rX(t) \coloneqq \sqrt{E_\rX(u(t)) + a_\rX}
\end{equation}
for some real numbers $a_\rX$.
In fact, the above decomposition of $E$ is always possible (\cref{rem:decomposition}), and thus we can address general (possibly unbounded) energy functionals by this approach.

We then address the second problem: what is a natural way to construct SAV schemes?
To answer this question, we begin with the modified energy
\begin{equation}
    \tE(u, r_\rL, r_\rU) = \frac{1}{2} \inpr{u}{Lu} + r_\rL^2 - r_\rU^2,
\end{equation}
which is a quadratic function.
The gradient of $\tE$ is then $\nabla \tE = [Lu, 2r_\rL, -2r_\rU]^T$.
With these notations, we will show that the equation of the SAV approach (such as \eqref{eq:AC_SAV}) is expressed as the gradient system for the modified functional $\tE$.
Further, this modified system inherits the conservation/dissipation property (i.e., variational structure) of the original gradient system \eqref{eq:variationalPDEs} (see~\cref{lem:operator-sav}).
From this point of view, existing SAV schemes are interpreted as a kind of discrete gradient method.

It is advantageous to keep the modified energy quadratic for the conservative case.
Indeed, it is known that a certain class of Runge-Kutta methods preserves quadratic invariants~\cite{HLW2010}.
Therefore, it is easy to construct higher-order time marching methods in the framework of Runge-Kutta methods and the SAV aproach, with the aid of our gradient system expression.

The remainder of this paper is organized as follows. 
In \cref{sec:grad}, we present the novel gradient system expression of the SAV approach, and show that the novel gradient system inherits the variational structure of the original system.
Then, we propose second and fourth order schemes in the framework of our expression in \cref{sec:proposed}.
Finally, in \cref{sec:ne}, we present some numerical examples of our scheme to confirm the efficiency and we conclude this paper in \cref{sec:cr}.

\section{Gradient system expression of SAV}
\label{sec:grad}

In this section, we present a novel interpretation of the SAV approach, namely, gradient system expression.
We state our result in an abstract setting.

Let $V$ be a Hilbert space over $\KK$ equipped with an inner-product $\inpr{\cdot}{\cdot}$, 
where $\KK=\RR$ or $\KK=\CC$, and $V^*$ be its dual.
Let $E$ be a $\KK$-valued $C^1$ function defined on some open subset of $V$, which is possibly unbounded, and $\nabla E(u) \colon V \to \KK$ be the Fr\'echet derivative of $E$ at $u \in V$.
We further let $ \Skew $ be a skew-adjoint or negative semidefinite linear operator on $V$ (not necessarily bounded).
Then, we consider the gradient system of the form
\begin{equation}
    \dot{u}(t) = \Skew \nabla E(u(t)), \qquad t \in (0,T),
    \label{eq:abstract-gf}
\end{equation}
where $u \colon (0,T) \to V$ is an unknown function and $T \in (0,\infty]$.
We assume that all of the terms in \eqref{eq:abstract-gf} is well-defined.

We present SAV approach for the general gradient flow \eqref{eq:abstract-gf}.
To begin with, we make an assumption on the energy functional $E$.
\begin{assumption}\label{assumption}
The energy functional $E$ can be decomposed into three parts as
\begin{equation}
    E(u) = \frac{1}{2} \inpr{u}{Lu} + E_\rL(u) - E_\rU(u), 
    \label{eq:assumption}
\end{equation}
where $L$ is a linear, self-adjoint, and positive semidefinite operator on $V$, and both $E_\rL$ and $E_\rU$ are lower-bounded $C^1$ functionals.
\end{assumption}

\begin{remark}\label{rem:decomposition}
Any function can be decomposed as in the above assumption.
Indeed, letting $\tilde{E}(u) = E(u) - \frac{1}{2}\inpr{u}{Lu}$, obviously we can see 
\begin{align}
    \tilde{E}(u) &= \left( \tilde{E}(u) + \tilde{E}(u)^2 \right) - \tilde{E}(u)^2 \\
                 &= \tilde{E}(u)^2 - \left( \tilde{E}(u)^2 - \tilde{E}(u) \right) \\
                 &= \left( \tilde{E}(u) + \tilde{E}(u)^4 \right) - \tilde{E}(u)^4 \\
                 &= \cdots.
\end{align}
Therefore, \cref{assumption} is just a notation.
Note that the decomposition is not unique.
\end{remark}

We now introduce two scalar auxiliary variables $r_\rL$ and $r_\rU$ by
\begin{equation}
    r_\rX(t) \coloneqq \sqrt{E_\rX(u(t)) + a_\rX},
\end{equation}
where $\rX = \rL$ or $\rU$ and $a_\rX > - \inf_u E_\rX(u)$ are real constants.
Then we define a modified energy functional by
\begin{equation}
    \tE(u, r_\rL, r_\rU) = \frac{1}{2} \inpr{u}{Lu} + r_\rL^2 - r_\rU^2,
\end{equation}
which is quadratic, and thus $\nabla \tE = [Lu, 2r_\rL, -2r_\rU]^T$.
Differentiating the scalar auxiliary variables, we have
\begin{equation}
    \dot{r}_\rX(t) = \frac{1}{ 2 \sqrt{E_\rX(u(t)) + a_\rX} } \inpr*{\nabla E_\rX(u(t))}{\dot{u}(t)} = \inpr{\phi_\rX}{\dot{u}(t)},
\end{equation}
where
\begin{equation}
    \phi_\rX(t) \coloneqq \frac{1}{ 2 \sqrt{E_\rX(u(t)) + a_\rX} } \nabla E_\rX(u(t)) \quad \in V^* \cong V.
\end{equation}
Although $\phi_\rX$ depends on $u$, we abbreviate the dependency to simplify the notation.
With this notation, the time derivative of the modified energy functional becomes
\begin{equation}
    \frac{\rd}{\rd t} \tE(u, r_\rL, r_\rU) = \inpr*{Lu + 2 r_\rL \phi_\rL - 2 r_\rU \phi_\rU}{\dot{u}}.
\end{equation}

According to the above observation, we consider the following equation
\begin{equation}
    \begin{dcases}
        \dot{u}(t) = \Skew \paren*{Lu + 2 r_\rL \phi_\rL - 2 r_\rU \phi_\rU}, \\
        \dot{r}_\rL(t) = \inpr{\phi_\rL}{\dot{u}(t)}, \\
        \dot{r}_\rU(t) = \inpr{\phi_\rU}{\dot{u}(t)}.
    \end{dcases}
    \label{eq:abstract-SAV}
\end{equation}
We show that the system \eqref{eq:abstract-SAV} can be expressed as a gradient system.
For each time $t$, define a linear operator $\Phi_\rX(t) \in V^*$ by $\Phi_\rX(t)(v) = \inpr*{\phi_\rX(t)}{v}$ $(v \in V)$. 
The adjoint operator of $\Phi_\rX$ is a multiplication operator $\Phi^*_\rX(t) (r) = r \phi_\rX(t)$.
Then, \eqref{eq:abstract-SAV} becomes
\begin{equation}
    \begin{bmatrix}
        I & 0 & 0 \\
        - \Phi_\rL & 1 & 0 \\
        - \Phi_\rU & 0 & 1 \\
    \end{bmatrix}
    \frac{\rd}{\rd t} 
    \begin{bmatrix} u \\ r_\rL \\ r_\rU \end{bmatrix}
    =
    \begin{bmatrix}
        \Skew & 0 & 0 \\
        0 & 0 & 0 \\
        0 & 0 & 0
    \end{bmatrix}
    \begin{bmatrix}
        I & \Phi^*_\rL & \Phi^*_\rU \\
        0 & 1 & 0 \\
        0 & 0 & 1
    \end{bmatrix}
    \nabla \tE(u, r_\rL, r_\rU),
\end{equation}
where $I$ is the identity operator on $V$.
Since it is easy to see that
\begin{equation}
    \begin{bmatrix}
    I & 0 & 0 \\
    - \Phi_\rL & 1 & 0 \\
    - \Phi_\rU & 0 & 1 \\
    \end{bmatrix}^{-1}
    =
    \begin{bmatrix}
        I & 0 & 0 \\
        \Phi_\rL & 1 & 0 \\
        \Phi_\rU & 0 & 1 \\
    \end{bmatrix},
\end{equation}
we obtain
\begin{equation}
    \frac{\rd}{\rd t} 
    \begin{bmatrix} u \\ r_\rL \\ r_\rU \end{bmatrix}
    =
    \begin{bmatrix}
        I & 0 & 0 \\
        \Phi_\rL & 1 & 0 \\
        \Phi_\rU & 0 & 1 \\
    \end{bmatrix}
    \begin{bmatrix}
        \Skew & 0 & 0 \\
        0 & 0 & 0 \\
        0 & 0 & 0
    \end{bmatrix}
    \begin{bmatrix}
        I & \Phi^*_\rL & \Phi^*_\rU \\
        0 & 1 & 0 \\
        0 & 0 & 1
    \end{bmatrix}
    \nabla \tE(u, r_\rL, r_\rU).
    \label{eq:abstract-SAV-gf}
\end{equation}
Hence, letting $Z = V \times \RR \times \RR$, $z(t) = [u(t), r_\rL(t), r_\rU(t)]^T$, and 
\begin{equation}
    \cL(u) = 
    \begin{bmatrix}
        I & 0 & 0 \\
        \Phi_\rL(u) & 1 & 0 \\
        \Phi_\rU(u) & 0 & 1 \\
    \end{bmatrix}
    \begin{bmatrix}
        \Skew & 0 & 0 \\
        0 & 0 & 0 \\
        0 & 0 & 0
    \end{bmatrix}
    \begin{bmatrix}
        I & \Phi^*_\rL(u) & \Phi^*_\rU(u) \\
        0 & 1 & 0 \\
        0 & 0 & 1
    \end{bmatrix}
    \colon Z \to Z,
    \quad u \in V,
\end{equation}
we can write \eqref{eq:abstract-SAV-gf} as
\begin{equation}
    \dot{z}(t) = \cL(u(t)) \nabla \tE(z(t)),
    \label{eq:abstract-SAV-gf-short}
\end{equation}
which is nothing but a gradient system in $Z$.
In fact, the system \eqref{eq:abstract-SAV-gf-short} inherits the structure of the original gradient system \eqref{eq:abstract-gf}.

\begin{lemma}\label{lem:operator-sav}
If $\Skew$ is skew-symmetric, then $\cL(u)$ is as well.
If $\Skew$ is negative semidefinite, then $\cL(u)$ is as well.
\end{lemma}

\begin{proof}
The statement is obvious by the definition of $\cL(u)$.
\end{proof}

\begin{remark}
The expression \eqref{eq:abstract-SAV-gf-short} (or equivalently \eqref{eq:abstract-SAV-gf}) is generalization of \eqref{eq:AC_SAV}.
Indeed, if $\Skew = -I$, $L=-\Delta$, $E_\rL(u) = \int_\Omega F(u) dx$, and $E_\rU(u) = 0$, we obtain \eqref{eq:AC_SAV}.
\end{remark}

\begin{remark}
The system \eqref{eq:abstract-SAV-gf-short} is easily obtained by substituting the first equation of \eqref{eq:abstract-SAV} into the second and the third ones.
However, it is not clear that the resulting equation by the naive calculation inherits the variational structure of the original equation.
Hence, not the system \eqref{eq:abstract-SAV-gf-short} itself but the above procedure is significant to see the variational structure of \eqref{eq:abstract-SAV}.
\end{remark}

\section{Proposed schemes}
\label{sec:proposed}

The new expression \eqref{eq:abstract-SAV-gf-short} gives insights to construct structure-preserving schemes for the gradient system \eqref{eq:abstract-gf}.
In this section, we propose a unified approach to construct numerical schemes for gradient systems based on \eqref{eq:abstract-SAV-gf-short}.
Let $\Dt>0$ be time increment, $t_n = n\Dt$, $z^n \approx z(t_n)$, and $z^n = [u^n, r^n_\rL, r^n_\rU]^T$.
We use the same notation as in the previous section.

\subsection{Second order scheme}

Let $z^{n+1/2} = (z^{n+1} + z^n)/2$.
Then, usual Crank-Nicolson scheme for \eqref{eq:abstract-SAV-gf-short} is
\begin{equation}
    \frac{z^{n+1} - z^n}{\Dt} = \cL\paren*{u^{n+1/2}} \nabla \tE \paren*{z^{n+1/2}}, 
\end{equation}
which is a nonlinear scheme.
We then replace the first $u^{n+1/2}$ by any vector that is (locally) $O(\Dt^2)$-approximation of $u(t_n + \frac{\Dt}{2})$, and obtain a Crank-Nicolson-type SAV scheme.

\begin{scheme}\label{scheme:sav-cn}
For given $z^n \in Z$, find $z^{n+1} \in Z$ that satisfies
\begin{equation}
    \frac{z^{n+1} - z^n}{\Dt} = \cL\paren*{\bu^{n+1/2}} \nabla \tE \paren*{z^{n+1/2}},
    \label{eq:sav-cn}
\end{equation}
where $\bu^{n+1/2}$ is any (locally) $O(\Dt^2)$-approximation of $u^{n+1/2}$.
\end{scheme}

For example, we can choose 
\begin{equation}
    \bu^{n+1/2} =
    \begin{dcases}
        u^0, & (n=0), \\
        \frac{3u^n - u^{n-1}}{2}, & (n \ge 1)
    \end{dcases}
\end{equation}
as in the literature.
Alternatively, we can determine $\bu^{n+1/2}$ by the forward Euler or exponential Euler method with time increment $\Dt/2$.
Regardless of the choice of $\bu^{n+1/2}$, \cref{scheme:sav-cn} has energy dissipation/preservation property
for the modified discrete energy functional defined by
\begin{equation}
    \tE^n \coloneqq \frac{1}{2} \inpr*{u^n}{Lu^n} + \paren*{r_\rL^n}^2 - \paren*{r_\rU^n}^2,
    \label{eq:disc-mod-Hamil}
\end{equation}
where $z^n = [u^n, r_\rL^n, r_\rU^n]^T \in Z$ is the solution of \eqref{eq:sav-cn}.

\begin{lemma}\label{lem:sav-cn}
    Let $z^n = [u^n, r_\rL^n, r_\rU^n] \in Z$ be the solution of \eqref{eq:sav-cn}.
    If $\Skew$ is skew-symmetric, then the modified energy $\tE^n$ is preserved, namely, $\tE^{n+1} = \tE^n$.
    If $\Skew$ is negative semidefinite, then the modified energy $\tE^n$ is dissipative, namely, $\tE^{n+1} \le \tE^n$.
\end{lemma}

\begin{proof}
We show energy preservation property only. The proof of the energy dissipation is the same.
By the symmetry of $L$, we have
\begin{align}
    \frac{\tE^{n+1} - \tE^n}{\Dt} 
    &= \inpr*{\frac{u^{n+1} - u^n}{\Dt}}{L u^{n+1/2}} + \frac{r^{n+1}_\rL - r^n_\rL}{\Dt} \cdot 2 r^{n+1/2}_\rL - \frac{r^{n+1}_\rU - r^n_\rU}{\Dt} \cdot 2 r^{n+1/2}_\rU \\
    &= \left[ \frac{z^{n+1} - z^n}{\Dt}, \nabla\tE(z^{n+1/2}) \right]_Z,
\end{align}
where $[\cdot, \cdot]_Z$ is the inner-product in $Z = V \times \RR \times \RR$ induced from the direct product.
Since $z^n$ satisfies \eqref{eq:sav-cn}, we obtain
\begin{equation}
    \frac{\tE^{n+1} - \tE^n}{\Dt} = \left[ \cL(\bu^{n+1/2}) \nabla \tE(z^{n+1/2}), \nabla\tE(z^{n+1/2}) \right]_Z.
\end{equation}
This implies $\tE^{n+1} = \tE^n$ because $\cL(\bu^{n+1/2}) $ is skew-symmetric from \cref{lem:operator-sav} for any $\bu^{n+1/2}$.
\end{proof}

We show that it is essentially sufficient to solve a partial differential equation with constant coefficients three times to get the solution of \cref{scheme:sav-cn} at each time step, as in the original SAV schemes.
It is clear that \eqref{eq:sav-cn} is equivalent to
\begin{equation}
    \frac{1}{\Dt}
    \begin{bmatrix}
        I & 0 & 0 \\
        - \Phi_\rL & 1 & 0 \\
        - \Phi_\rU & 0 & 1 \\
    \end{bmatrix}
    \begin{bmatrix} u^{n+1} - u^n \\ r_\rL^{n+1} - r_\rL^n \\ r_\rU^{n+1} - r_\rU^n \end{bmatrix}
    =
    \begin{bmatrix}
        \Skew & \Skew \Phi^*_\rL & \Skew \Phi^*_\rU \\
        0 & 0 & 0 \\
        0 & 0 & 0
    \end{bmatrix}
    \begin{bmatrix} L u^{n+1/2} \\ 2 r_\rL^{n+1/2} \\ -2 r_\rU^{n+1/2} \end{bmatrix},
\end{equation}
where $\Phi_\rX = \Phi_\rX\paren*{\bar{u}^{n+1/2}}$.
This yields
\begin{equation}
    \begin{bmatrix}
        I - \frac{\Dt}{2} \Skew L \hsp & - \Dt \Skew \Phi^*_\rL \hsp & \Dt \Skew \Phi^*_\rU \\
        - \Phi_\rL & 1 & 0 \\
        - \Phi_\rU & 0 & 1 
    \end{bmatrix}
    \begin{bmatrix} u^{n+1} \\ r_\rL^{n+1} \\ r_\rU^{n+1}  \end{bmatrix}
    =
    \begin{bmatrix}
        I + \frac{\Dt}{2} \Skew L \hsp & \Dt \Skew \Phi^*_\rL \hsp & - \Dt \Skew \Phi^*_\rU \\
        - \Phi_\rL & 1 & 0 \\
        - \Phi_\rU & 0 & 1 
    \end{bmatrix}
    \begin{bmatrix} u^n \\ r_\rL^n \\ r_\rU^n \end{bmatrix}.
\end{equation}
Assume that the operator $J \coloneqq I - \frac{\Dt}{2} \Skew L$ is invertible.
Then, multiplying the last equation by the matrix
\begin{equation}
    \begin{bmatrix}
        I & 0 & 0 \\
        \Phi_\rL J^{-1} & 1 & 0 \\
        \Phi_\rU J^{-1} & 0 & 1
    \end{bmatrix}
\end{equation}
from the left, we have
\begin{multline}
    \begin{bmatrix}
        J \hsp & - \Dt \Skew \Phi^*_\rL                                & \hsp \Dt \Skew \Phi^*_\rU \\[0.5ex]
        0 \hsp & -\Dt \Phi_\rL J^{-1} \Skew \Phi_\rL^* + 1             & \hsp \Dt \Phi_\rL J^{-1} \Skew \Phi_\rU^* \hphantom{{}+1} \\[0.5ex]
        0 \hsp & -\Dt \Phi_\rU J^{-1} \Skew \Phi_\rL^* \hphantom{{}+1} & \hsp \Dt \Phi_\rU J^{-1} \Skew \Phi_\rU^* + 1
    \end{bmatrix}
    \begin{bmatrix} u^{n+1} \\ r_\rL^{n+1} \\ r_\rU^{n+1}  \end{bmatrix} \\
    =
    \begin{bmatrix}
        I + \frac{\Dt}{2} \Skew L   \hsp & \Dt \Skew \Phi^*_\rL                                 & \hsp 
        - \Dt \Skew \Phi^*_\rU \\[0.5ex]
        2 \Phi_\rL \paren*{J^{-1} - I} \hsp & \Dt \Phi_\rL J^{-1} \Skew \Phi_\rL^* + 1             & \hsp - \Dt \Phi_\rL J^{-1} \Skew \Phi_\rU^* \hphantom{{}+1} \\[0.5ex]
        2 \Phi_\rU \paren*{J^{-1} - I} \hsp & \Dt \Phi_\rU J^{-1} \Skew \Phi_\rL^* \hphantom{{}+1} & \hsp - \Dt \Phi_\rU J^{-1} \Skew \Phi_\rU^* + 1
    \end{bmatrix}
    \begin{bmatrix} u^n \\ r_\rL^n \\ r_\rU^n \end{bmatrix},
\end{multline}
which is just block Gauss elimination.
Therefore, we can decompose the equation into two parts. 
The first one is the two-dimensional linear system
\begin{multline}
    \begin{bmatrix}
        -\Dt \Phi_\rL J^{-1} \Skew \Phi_\rL^* + 1             & \hsp \Dt \Phi_\rL J^{-1} \Skew \Phi_\rU^* \hphantom{{}+1} \\[0.5ex]
        -\Dt \Phi_\rU J^{-1} \Skew \Phi_\rL^* \hphantom{{}+1} & \hsp \Dt \Phi_\rU J^{-1} \Skew \Phi_\rU^* + 1
    \end{bmatrix}
    \begin{bmatrix} r_\rL^{n+1} \\ r_\rU^{n+1}  \end{bmatrix} \\
    =
    \begin{bmatrix}
        2 \Phi_\rL \paren*{J^{-1} - I} \hsp & \Dt \Phi_\rL J^{-1} \Skew \Phi_\rL^* + 1             & \hsp - \Dt \Phi_\rL J^{-1} \Skew \Phi_\rU^* \hphantom{{}+1} \\[0.5ex]
        2 \Phi_\rU \paren*{J^{-1} - I} \hsp & \Dt \Phi_\rU J^{-1} \Skew \Phi_\rL^* \hphantom{{}+1} & \hsp - \Dt \Phi_\rU J^{-1} \Skew \Phi_\rU^* + 1
    \end{bmatrix}
    \begin{bmatrix} u^n \\ r_\rL^n \\ r_\rU^n \end{bmatrix},
\end{multline}
which is equivalent to
\begin{multline}
    \begin{bmatrix}
        -\Dt \inpr*{\phi_\rL}{J^{-1} \Skew \phi_\rL} + 1             & \hsp \Dt \inpr*{\phi_\rL}{J^{-1} \Skew \phi_\rU} \hphantom{{}+1} \\[0.5ex]
        -\Dt \inpr*{\phi_\rU}{J^{-1} \Skew \phi_\rL} \hphantom{{}+1} & \hsp \Dt \inpr*{\phi_\rU}{J^{-1} \Skew \phi_\rU} + 1
    \end{bmatrix}
    \begin{bmatrix} r_\rL^{n+1} \\ r_\rU^{n+1}  \end{bmatrix} \\[1ex]
    =
    \begin{bmatrix}
        2 \inpr*{\phi_\rL}{ \paren*{J^{-1} - I} u^n} + \left[ \Dt \inpr*{\phi_\rL}{J^{-1} \Skew \phi_\rL} + 1 \right] r^n_\rL - \Dt \inpr*{\phi_\rL}{J^{-1} \Skew \phi_\rU} r^n_\rU \\[1ex]
        2 \inpr*{\phi_\rU}{ \paren*{J^{-1} - I} u^n} + \Dt \inpr*{\phi_\rU}{J^{-1} \Skew \phi_\rL} r^n_\rL - \left[ \Dt \inpr*{\phi_\rU}{J^{-1} \Skew \phi_\rU} - 1 \right] r^n_\rU
    \end{bmatrix}.
    \label{eq:sav-cn-forward}
\end{multline}
The second one is 
\begin{equation}
    J u^{n+1} - \Dt \Skew \Phi^*_\rL r^{n+1}_\rL + \Dt \Skew \Phi^*_\rU r^{n+1}_\rU 
    = \paren*{I + \frac{\Dt}{2} \Skew L} u^n + \Dt \Skew \Phi^*_\rL r^n_\rL - \Dt \Skew \Phi^*_\rU r^n_\rU, 
\end{equation}
which is, in this case, equivalent to
\begin{equation}
    u^{n+1} = \paren*{2J^{-1} - I} u^n + \Dt \paren*{r^{n+1}_\rL + r^n_\rL} J^{-1} D \phi_\rL - \Dt \paren*{r^{n+1}_\rU + r^n_\rU} J^{-1} D \phi_\rU.
    \label{eq:sav-cn-backward}
\end{equation}

We can summarize the above procedure as the following algorithm.
\begin{enumerate}
    \item Compute $J^{-1}u^n$, $J^{-1} D \phi_\rL \paren*{\bar{u}^{n+1/2}}$, and $J^{-1} D \phi_\rU \paren*{\bar{u}^{n+1/2}}$.
    \item Compute $r^{n+1}_\rL$ and $r^{n+1}_\rU$ by solving \eqref{eq:sav-cn-forward}.
    \item Compute $u^{n+1}$ by using \eqref{eq:sav-cn-backward}.
\end{enumerate}
For the first step, we should solve the equation of the form $J w = f$ three times.
Therefore, it is sufficient to solve a partial differential equation with constant coefficients to get the solution of \cref{scheme:sav-cn} at each time step.

\subsection{Fourth order scheme for conservative systems}

In this subsection, we focus on the conservative systems, namely, suppose $\Skew$ is skew-adjoint.
According to the expression \eqref{eq:abstract-SAV-gf-short}, we can construct higher order schemes via Runge-Kutta methods.
In the present paper, we present a fourth order SAV scheme.
General theory to construct SAV Runge-Kutta schemes will be presented elsewhere.

Let us consider a two-stage Runge-Kutta method for \eqref{eq:abstract-SAV-gf-short} as follows:
\begin{equation}
    \left\{
    \begin{aligned}
        Z^n_i &= z^n + \Dt \sum_{j=1}^2 a_{ij} \cL \paren*{U^n_j} \nabla \tE \paren*{Z^n_j}, \quad i=1,2,\\
        z^{n+1} &= z^n + \Dt \sum_{j=1}^2 b_j  \cL \paren*{U^n_j} \nabla \tE \paren*{Z^n_j},
    \end{aligned}
    \right.
    \quad 
    Z^n_j = \begin{bmatrix} U^n_j \\ R^n_{\rL,j} \\ R^n_{\rU,j} \end{bmatrix},
    \label{eq:RK}
\end{equation}
where $a_{i,j} \in \RR$ and $b_j \in \RR$.
The Runge-Kutta method \eqref{eq:RK} is called canonical if 
\begin{equation}
    b_i b_j = b_i a_{ij} + b_j a_{ji},
    \label{eq:cannonical-RK}
\end{equation}
and it is known that canonical Runge-Kutta method preserves the energy functional $\tE$ if it is quadratic (cf.~\cite[Theorem IV.2.2]{HLW2010}).

The typical example is the Runge-Kutta-Gauss-Legendre method with the Butcher tableau
\begin{equation}
    \begin{array}{c|c}
        c & A \\\hline 
          & b 
    \end{array}
    \quad = \quad 
    \begin{array}{c|cc}
        \frac{1}{2} - \frac{\sqrt{3}}{6} & \frac{1}{4} & \frac{1}{4} - \frac{\sqrt{3}}{6} \\[1ex]
        \frac{1}{2} + \frac{\sqrt{3}}{6} & \frac{1}{4} + \frac{\sqrt{3}}{6} & \frac{1}{4} \\[1ex]\hline 
                                         & \frac{1}{2} & \frac{1}{2} 
    \end{array},
\end{equation}
where $A=(a_{ij})$ and $b = (b_i)$.
Here, $c = (c_j)$ is unused because \eqref{eq:abstract-SAV-gf-short} is autonomous.
Although this method is fourth order, it is implicit and thus a nonlinear scheme.
In order to construct a linear scheme, we replace $U^n_j$ in $\cL$ by other explicitly determined vector $\bar{U}^n_j$ as follows:
\begin{equation}
    \left\{
    \begin{aligned}
        Z^n_i &= z^n + \Dt \sum_{j=1}^2 a_{ij} \cL \paren*{\bar{U}^n_j} \nabla \tE \paren*{Z^n_j}, \qquad i=1,2,\\
        z^{n+1} &= z^n + \Dt \sum_{j=1}^2 b_j  \cL \paren*{\bar{U}^n_j} \nabla \tE \paren*{Z^n_j}.
    \end{aligned}
    \right.
    \label{eq:modified-RK}
\end{equation}
Since $\tE$ is quadratic and the Runge-Kutta method is canonical, the modified method \eqref{eq:modified-RK} also preserves the energy functional $\tE^n = \tE\paren*{z^n}$.
\begin{lemma}\label{lem:sav-rk}
    Let $z^n$ be the solution of the \eqref{eq:modified-RK}.
    If $a_{ij}$ and $b_i$ satisfy the relation \eqref{eq:cannonical-RK}, then $\tE^{n+1} = \tE^n$ regardless of the choice of $\bar{U}^n_j$.
\end{lemma}
\begin{proof}
The proof is the same as in \cite[Theorem IV.2.2]{HLW2010}.
\end{proof}

Now, it is necessary to find the predetermined vectors $\bar{U}^n_j$.
Since $U^n_j$ in \eqref{eq:RK} is locally $O(\Dt^3)$-approximation of $u(t_n + c_j \Dt)$, 
it is required that $\bar{U}^n_j$ approximates the same vector with the same order.
To construct such vectors, we consider two methods.

The one is to use another explicit five-stage Runge-Kutta method for the original equation \eqref{eq:abstract-gf}, which is deduced from the framework of partitioned Runge-Kutta methods.
That is, define $\bar{Y}^n_j$ by
\begin{equation}
    \bar{Y}^n_i = u^n + \Dt \sum_{j=1}^5 \bar{a}_{ij} \Skew \nabla E\paren*{\bar{Y}^n_j},
    \label{eq:auxiliary-RK}
\end{equation}
where
\begin{equation}
    \paren{\bar{a}_{ij}} = 
    \begin{bmatrix}
        0 & 0 & 0 & 0 & 0 \\[1ex]
        \frac{1}{4} & 0 & 0 & 0 & 0 \\[1ex]
        0 & \frac{1}{2} & 0 & 0 & 0 \\[1ex]
        \frac{1}{6} & 0 & \frac{1}{3} - \frac{\sqrt{3}}{6} & 0 & 0 \\[1ex]
        \frac{1}{6} & 0 & \frac{1}{3} + \frac{\sqrt{3}}{6} & 0 & 0 
    \end{bmatrix}.
    \label{eq:auxiliary-RK-coeff}
\end{equation}
Then, $\bar{U}^n_1 \coloneqq \bar{Y}^n_4$ and $\bar{U}^n_2 \coloneqq \bar{Y}^n_5$ are expected to fulfill the requirements. This can be proved in a general framework, which will be presented elsewhere.

The second choice is to utilize the three-stage exponential Runge-Kutta method.
This strategy is efficient for the PDE case.
Assume that the original equation \eqref{eq:abstract-gf} can be written as
\begin{equation}
    \dot{u}(t) = Au + g(u)
\end{equation}
where $A$ is a linear operator on $V$ and $g$ is some nonlinear function.
Then, an exponential Runge-Kutta method for \eqref{eq:abstract-gf} is described as follows:
\begin{equation}
    \left\{
    \begin{aligned}
        U^{n+1}_i &= u^n + \Delta t \sum_{j=1}^3 \tilde{a}_{ij}(\Delta t A) \paren*{g(U^{n+1}_j) + A u^n}, \qquad i=1,2,3, \\
        u^{n+1} &= u^n + \Delta t \sum_{j=1}^s \tilde{b}_j(\Delta t A) \paren*{g(U^{n+1}_j) + A u^n} ,
    \end{aligned}
    \right.
\end{equation}
where $\tilde{a}_{ij}$ and $\tilde{b}_j$ are functions defined as follows:
\begin{equation}
    \begin{array}{c|c}
        \tilde{c} & \tilde{A} \\\hline & \tilde{b}
    \end{array}
    = 
    \begin{array}{c|ccc}
        0 & 0 & 0 & 0 \\[1ex]
        \frac{1}{3} & \frac{1}{3} \varphi_{1,3} & 0 & 0 \\[1ex]
        \frac{2}{3} & \frac{2}{3} \varphi_{1,3} - \frac{4}{3} \varphi_{2,3} & \frac{4}{3} \varphi_{2,3} & 0 \\[1ex]\hline
        & \varphi_1 - \frac{3}{2} \varphi_2 & 0 & \frac{3}{2} \varphi_2
    \end{array}
    \label{eq:exp-RK-coeff}
\end{equation}
with
\begin{equation}
    \varphi_1(z) = \frac{e^z-1}{z} ,
    \quad 
    \varphi_2(z) = \frac{e^z-1-z}{z^2},
    \quad
    \varphi_{i,j}(z) = \varphi_i(\tilde{c}_j z).
\end{equation}
Let us write this integrator by $\operatorname{ERK}(\Delta t)$. 
Then, $\bar{U}^n_j \coloneqq \operatorname{ERK}(c_j \Delta t) u^n$ ($j=1,2$) is locally $O(\Delta t^3)$-approximation of $u(t_n + c_j \Delta t)$, which is what we desired.

Conclusively, we obtain the fourth order SAV-Runge-Kutta scheme as follows.
\begin{scheme}\label{scheme:sav-rk}
For given $z^n \in Z$, find $z^{n+1} \in Z$ that satisfies \eqref{eq:modified-RK},
where $\bar{U}^n_j$ is determined by the above methods.
\end{scheme}

\section{Numerical Experiments}
\label{sec:ne}

\subsection{Kepler problem}

We first consider, as a toy problem, the Kepler ODE model
\begin{equation}
    \frac{d}{dt} 
    \begin{bmatrix}
        x \\ y \\ u \\ v
    \end{bmatrix}
    =
    \begin{bmatrix}
        u \\ v \\ -x/r^3 \\ -y/r^3
    \end{bmatrix},
    \label{eq:kepler}
\end{equation}
where $r = \sqrt{x^2+y^2}$.
The corresponding energy functional is
\begin{equation}
    E(x,y,u,v) = \frac{u^2+v^2}{2} - \frac{1}{r}
\end{equation}
and we can rewrite equation \eqref{eq:kepler} to the form \eqref{eq:abstract-gf} on $V = \mathbb{R}^4$ with
\begin{equation}
    D = \begin{bmatrix}
        O_2 & I_2 \\ -I_2 & O_2
    \end{bmatrix},
\end{equation}
where $I_2$ and $O_2$ are the identity and zero matrices in $\mathbb{R}^{2 \times 2}$, respectively.
Let $w = [x,y,u,v]^T \in V$. 
We decompose $E$ as
\begin{equation}
    E(w) = \frac{|w|^2}{2} + E_\rL(w) - E_\rU(w),
    \qquad
    E_\rL(w) = 0,
    \quad
    E_\rU(w) = \frac{1}{r},
\end{equation}
which corresponds to \cref{assumption} with $L=I \in \mathbb{R}^{4 \times 4}$.

We then consider three schemes: 
the Crank-Nicolson-type scheme (\cref{scheme:sav-cn}) with $\bu^{n+1/2}$ determined by the extrapolation (SAV-CN-ext) and the forward Euler method with time increment $\Dt/2$ (SAV-CN-Euler);
the fourth order SAV-Runge-Kutta scheme (\cref{scheme:sav-rk}) with $\bar{U}^n_j$ determined by the explicit five-stage Runge-Kutta method with coefficients \eqref{eq:auxiliary-RK-coeff} (SAV-RK).
Throughout the experiments, we set $w(0) = [0.2,0,0,0.3]^T$, which implies the solution is periodic with period $T=2\pi$, and $a_\rU = 1$.

\subsubsection{Numerical results}

We first observe that the proposed schemes work well.
We set $\Dt = T/2^{10}$ and computed $N=10 \cdot 2^{10}$ steps (ten periods) for each scheme.
The results are plotted in \cref{fig:kepler-SAV-CN-ext,fig:kepler-SAV-CN-Euler,fig:kepler-SAV-RK}.
For each result, the left figure shows the orbit until $t=10T = 20\pi$, where small circles expresses the solution at $t=nT$ for $n=0,1,\dots,10$, and the right one shows the relative error of the modified energy functional $\tilde{E}^n$.

Figures \ref{fig:kepler-SAV-CN-ext}\subref{subfig:kepler-SAV-CN-ext-orbit} and \ref{fig:kepler-SAV-CN-Euler}\subref{subfig:kepler-SAV-CN-Euler-orbit} show the orbit of numerical solutions of the second order Crank-Nicolson-type schemes.
This result suggests that the choice of the predetermined vector $\bar{u}^n$ in \eqref{eq:sav-cn} may affects the accuracy.
\cref{fig:kepler-SAV-RK}\subref{subfig:kepler-SAV-RK-orbit} shows the orbit of numerical solutions of the fourth order scheme and the result suggests that this scheme has good accuracy.
Furthermore, Figures \ref{fig:kepler-SAV-CN-ext}\subref{subfig:kepler-SAV-CN-ext-mHamil}--\ref{fig:kepler-SAV-RK}\subref{subfig:kepler-SAV-RK-mHamil} show that the modified energy functional is numerically conserved, which support \cref{lem:sav-cn,lem:sav-rk}.

\begin{figure}
    \centering
    \begin{subfigure}{0.45\linewidth}
        \centering
        \includegraphics[page=3,width=\linewidth]{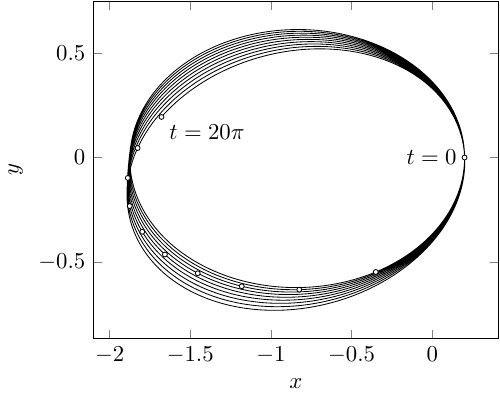}
        \caption{Orbit}
        \label{subfig:kepler-SAV-CN-ext-orbit}
    \end{subfigure}
    \hfill
    \begin{subfigure}{0.45\linewidth}
        \centering
        \includegraphics[page=6,width=\linewidth]{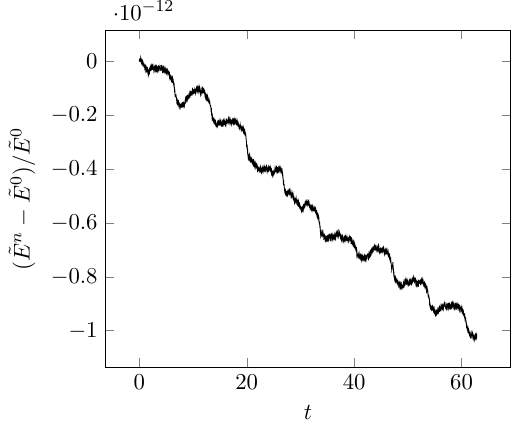}
        \caption{Modified energy functional}
        \label{subfig:kepler-SAV-CN-ext-mHamil}
    \end{subfigure}
    \caption{Numerical results for \eqref{eq:kepler} by SAV-CN-ext}
    \label{fig:kepler-SAV-CN-ext}
\end{figure}

\begin{figure}
    \centering
    \begin{subfigure}{0.45\linewidth}
        \centering
        \includegraphics[page=4,width=\linewidth]{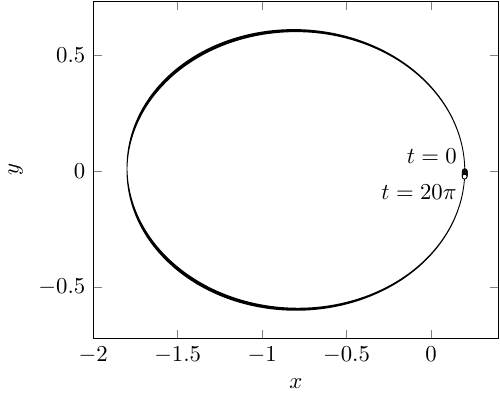}
        \caption{Orbit}
        \label{subfig:kepler-SAV-CN-Euler-orbit}
    \end{subfigure}
    \hfill
    \begin{subfigure}{0.45\linewidth}
        \centering
        \includegraphics[page=7,width=\linewidth]{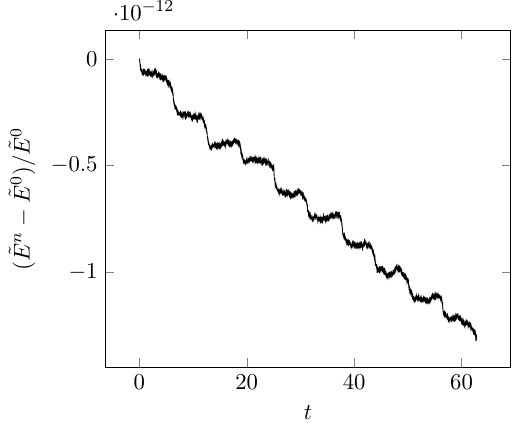}
        \caption{Modified energy functional}
        \label{subfig:kepler-SAV-CN-Euler-mHamil}
    \end{subfigure}
    \caption{Numerical results for \eqref{eq:kepler} by SAV-CN-Euler}
    \label{fig:kepler-SAV-CN-Euler}
\end{figure}

\begin{figure}
    \centering
    \begin{subfigure}{0.45\linewidth}
        \centering
        \includegraphics[page=5,width=\linewidth]{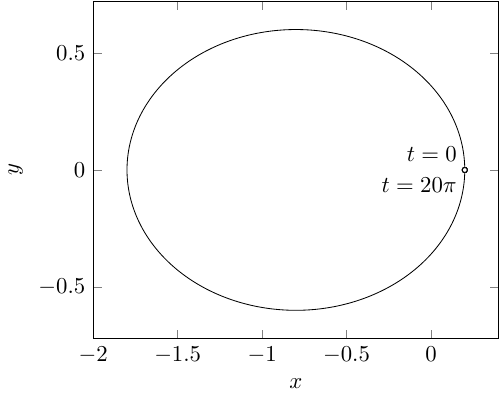}
        \caption{Orbit}
        \label{subfig:kepler-SAV-RK-orbit}
    \end{subfigure}
    \hfill
    \begin{subfigure}{0.45\linewidth}
        \centering
        \includegraphics[page=8,width=\linewidth]{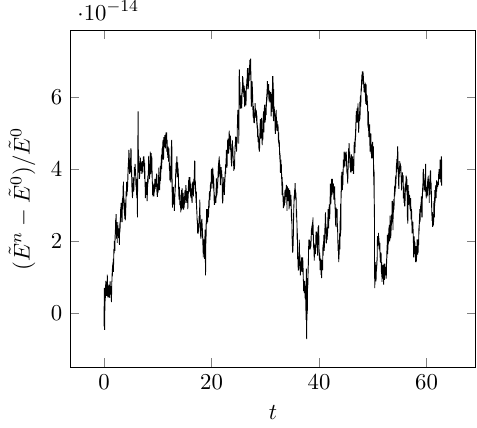}
        \caption{Modified energy functional}
        \label{subfig:kepler-SAV-RK-mHamil}
    \end{subfigure}
    \caption{Numerical results for \eqref{eq:kepler} by SAV-RK}
    \label{fig:kepler-SAV-RK}
\end{figure}

\subsubsection{Convergence rates}

We now consider the convergence rates numerically.
We set $\Dt = 2\pi/N_t$ for $N_t=2^i$, $i=7,8,\dots,20$ and computed the relative error
\begin{equation}
    \frac{\| w^{N_t} - w^0 \|_{l^\infty}}{\| w^0 \|_{l^\infty}}
\end{equation}
for the solution $w^n$ of each scheme.
We plotted the results in \cref{fig:kepler-rate}\subref{subfig:kepler-rate-sol}.
We also computed the original energy functional $E^n = E(w^n)$ and plotted the relative error
\begin{equation}
    \max_{0 \le n \le N_t} \frac{|E^n - E^0|}{|E^0|}
\end{equation}
in \cref{fig:kepler-rate}\subref{subfig:kepler-rate-Hamil}.
These results suggest that, for both the solution and the energy functional, the convergence rate is $O(\Dt^2)$ for the second order schemes and $O(\Dt^4)$ for the fourth order scheme, as expected.
Moreover, one can see that the error of $w$ of SAV-CN-ext is about $10^3$ times bigger than that of SAV-CN-Euler, which reflects the situation observed in Figures \ref{fig:kepler-SAV-CN-ext}\subref{subfig:kepler-SAV-CN-ext-orbit} and \ref{fig:kepler-SAV-CN-Euler}\subref{subfig:kepler-SAV-CN-Euler-orbit}.

\begin{figure}\centering
    \begin{subfigure}{0.45\linewidth}
        \centering
        \includegraphics[page=1,width=\linewidth]{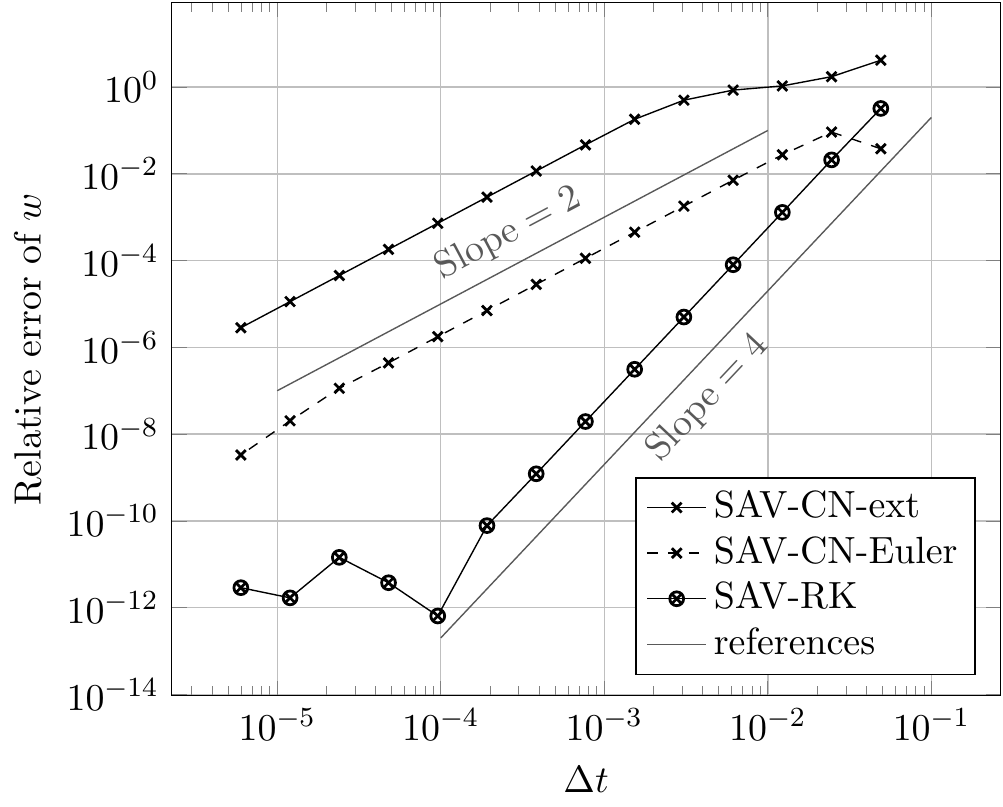}
        \caption{Relative error of the solution.}
        \label{subfig:kepler-rate-sol}
    \end{subfigure}
    \hfill
    \begin{subfigure}{0.45\linewidth}
        \centering
        \includegraphics[page=2,width=\linewidth]{figures-kepler.pdf}
        \caption{Relative error of the energy functional.}
        \label{subfig:kepler-rate-Hamil}
    \end{subfigure}
    \caption{Convergence rates of the proposed schemes for the Kepler problem \eqref{eq:kepler}. The gray-colored thin lines are drawn for references}
    \label{fig:kepler-rate}
\end{figure}

\subsection{Korteveg--de Vries equation}

We next consider the Korteveg--de Vries (KdV) equation 
\begin{equation}
    \frac{\partial u}{\partial t} = \frac{\partial}{\partial x} \paren*{ - \frac{\partial^2 u}{\partial x^2} - 3u^2}
    \qquad \text{in } (0,L) \times (0,T)
    \label{eq:kdv}
\end{equation}
with the periodic boundary condition.
The corresponding energy functional is
\begin{equation}
    E[u] = \int_0^L \paren*{ \frac{u_x^2}{2} - u^3} dx.
\end{equation}
We first discretize the equation spatially by the spectral difference method \cite{MSFM2002}.
Namely, for $N \in 2\mathbb{N}$, we set $\Delta x = L/N$, $x_j = j \Delta x$, and $u_j(t) \approx u(x_j,t)$, and define the operator $\delta \in \mathbb{R}^{N \times N}$ by  
\begin{equation}
    \delta u \coloneqq F^{-1} \Xi F u, \qquad u \in \mathbb{R}^N,
\end{equation}
where 
\begin{equation}
    F = \paren*{e^{-2\pi ijk/N}}_{0 \le j,k \le N-1} \in \mathbb{C}^{N \times N}
\end{equation}
is the discrete Fourier transform and 
\begin{equation}
    \Xi = \operatorname{diag}\paren*{0, 1, 2, \dots, \frac{N}{2} - 1, 0, - \paren*{\frac{N}{2}-1}, \dots, -2, -1}\frac{2 \pi i}{L}
    \in \mathbb{C}^{N \times N}.
\end{equation}
Although $F$ and $\Xi$ are complex matrices, it is easy to see that $\delta$ is a real matrix.
Then, we discretize the KdV equation \eqref{eq:kdv} spatially as
\begin{equation}
    u_t = \delta \paren*{ -\delta^2 u - 3u \odot u },
    \label{eq:disc-kdv}
\end{equation}
where $u = (u_j)_{j=0}^{N-1} \in \mathbb{R}^N$ and $\odot$ is the Hadamard product.
Then, the functional
\begin{equation}
    E[u] = \sum_{j=0}^{N-1} \paren*{\frac{1}{2}(\delta u)_j^2 - u_j^3} \Delta x
    \label{eq:disc-Hamil}
\end{equation}
is conserved.

Now, let us define $g_L$ and $g_U$ by
\begin{equation}
    g_L(u) = u^4 - u^3, \quad g_U(u) = u^4 \qquad (u \in \mathbb{R})
\end{equation}
and introduce the functionals
\begin{equation}
    E_\rX[u] \coloneqq \sum_{j=0}^{N-1} g_X(u_j) \Delta x \qquad \rX=\rL,\rU 
\end{equation}
for $u \in \mathbb{R}^N$. 
Then, we have the decomposition
\begin{equation}
    E[u] = \frac{1}{2} u^T (-\delta^2 u) \Delta x + E_\rL[u] - E_\rU[u],
\end{equation}
which satisfies \cref{assumption}, and thus, for $D=\delta$, $L=-\delta^2$, and $V = \mathbb{R}^{N \times N}$ with the inner product $\inpr{u}{v} \coloneqq u^T v \Delta x$, the spatially discretized KdV equation \eqref{eq:disc-kdv} is expressed by the gradient system \eqref{eq:abstract-gf}.
Hence we obtain the system \eqref{eq:abstract-SAV-gf-short} with some parameters $a_\rL$ and $a_\rU$, and we can apply our schemes that conserves the modified energy functional
\begin{equation}
    \tE^n = \frac{1}{2} \inpr{u^n}{-\delta^2 u^n} + \paren*{r_\rL^n}^2 - \paren*{r_\rU^n}^2.
    \label{eq:kdv-mHamil}
\end{equation}

We here consider three schemes: 
the Crank-Nicolson-type scheme (\cref{scheme:sav-cn}) with $\bu^{n+1/2}$ determined by the extrapolation (SAV-CN-ext) and the exponential Euler method with time increment $\Dt/2$ (SAV-CN-Euler);
the fourth order SAV-Runge-Kutta scheme (\cref{scheme:sav-rk}) with $\bar{U}^n_j$ determined by the three-stage exponential Runge-Kutta method with coefficients \eqref{eq:exp-RK-coeff} (SAV-RK).

Throughout the experiments, we set the initial function by
\begin{equation}
u(x,0) = u_0 + 2 \kappa^2 k^2 \operatorname{cn}^2\paren*{\kappa x | k}    
\end{equation}
for $k=\sqrt{0.1}$, $\kappa=1$, and $u_0 = 0$ so that the solution is the cnoidal wave
\begin{equation}
    u(x,t) = u_0 + 2 \kappa^2 k^2 \operatorname{cn}^2\paren*{\kappa(x-ct) | k},
    \qquad
    c = 6u_0 + 4(2k^2-1)\kappa^2,
    \label{eq:cnoidal}
\end{equation}
which is periodic in space with $p = 2K(k)/\kappa \approx 3.2249$ and in time with period $T = p/|c| \approx 1.0078$.
Here, $\operatorname{cn}$ is one of the Jacobi elliptic functions and $K(k)$ is the complete elliptic integral of the first kind.
We also set the length of the spatial interval by $L=p$, the frequency of the spectral difference method by $N=16$, and the parameters for the auxiliary variables by $a_\rL = a_\rU = 1$.

\subsubsection{Conservation law}

We first observe that the proposed schemes conserves the modified energy functional $\tilde{E}^n$.
We set $\Dt = T/2^{10}$ and computed one period for each scheme.
We plotted the relative error of the modified energy functional \eqref{eq:kdv-mHamil}.
The results are plotted in \cref{fig:kdv-SAV-mHamil} and one can observe that \eqref{eq:kdv-mHamil} is conserved numerically, which support \cref{lem:sav-cn,lem:sav-rk}.

\begin{figure}
    \centering
    \begin{subfigure}{0.32\linewidth}
        \centering
        \includegraphics[page=3,width=\linewidth]{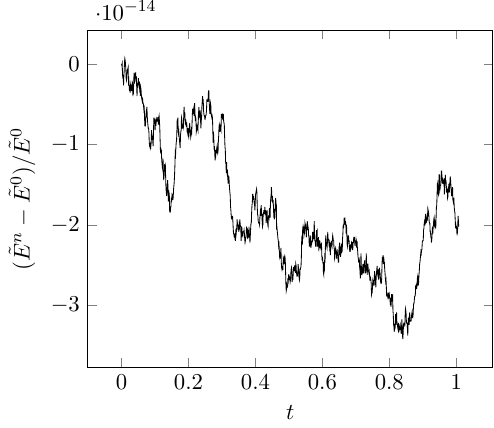}
        \caption{SAV-CN-ext}
    \end{subfigure}
    \hfill
    \begin{subfigure}{0.32\linewidth}
        \centering
        \includegraphics[page=4,width=\linewidth]{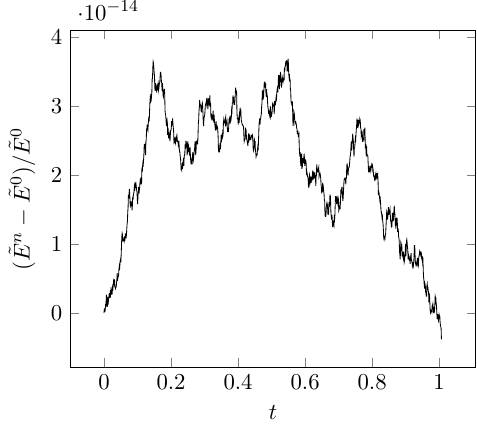}
        \caption{SAV-CN-Euler}
    \end{subfigure}
    \hfill
    \begin{subfigure}{0.32\linewidth}
        \centering
        \includegraphics[page=5,width=\linewidth]{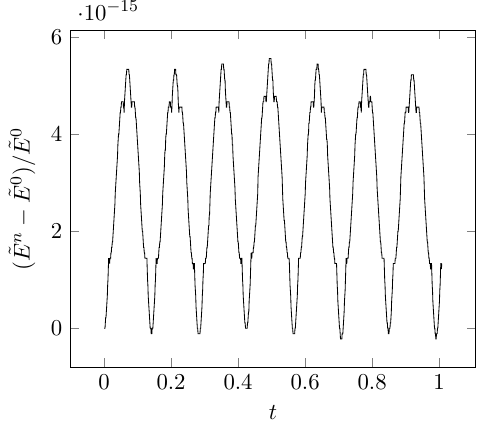}
        \caption{SAV-CN-RK}
    \end{subfigure}
    \caption{Relative error of the modified energy functional for the spatially discretized KdV equation \eqref{eq:disc-Hamil}}
    \label{fig:kdv-SAV-mHamil}
\end{figure}

\subsubsection{Convergence rate}

We observe the convergence rate numerically.
We set $\Dt = T/N_t$ for $N_t=2^i$, $i=3,4,\dots,20$ and computed the relative error
\begin{equation}
    \frac{\| u^{N_t} - u^0 \|_{l^\infty}}{\| u^0 \|_{l^\infty}}
\end{equation}
for the solution $u^n$ of each scheme.
We plotted the results in \cref{fig:kdv-rate}\subref{subfig:kdv-rate-sol}.
This result suggests that the convergence rate for the solution is $O(\Dt^2)$ for the second order schemes and $O(\Dt^4)$ for the fourth order scheme, as expected.

We also computed the spatially discretized original energy functional  $E^n = E(u^n)$ defined by \eqref{eq:disc-mod-Hamil} and plotted the relative error
\begin{equation}
    \max_{0 \le n \le N_t} \frac{|E^n - E^0|}{|E^0|}
\end{equation}
in \cref{fig:kdv-rate}\subref{subfig:kdv-rate-Hamil}.
In contrast to the case of $u$, the convergence rates of $E^n$ for SAV-Euler and SAV-RK are better than expected.
Although it is not clear why these phenomena occur, we can expect that the upper bound for convergence rate is $O(\Dt^2)$ or $O(\Dt^4)$ for each scheme.

\begin{figure}\centering
    \begin{subfigure}{0.45\linewidth}
        \centering
        \includegraphics[page=1,width=\linewidth]{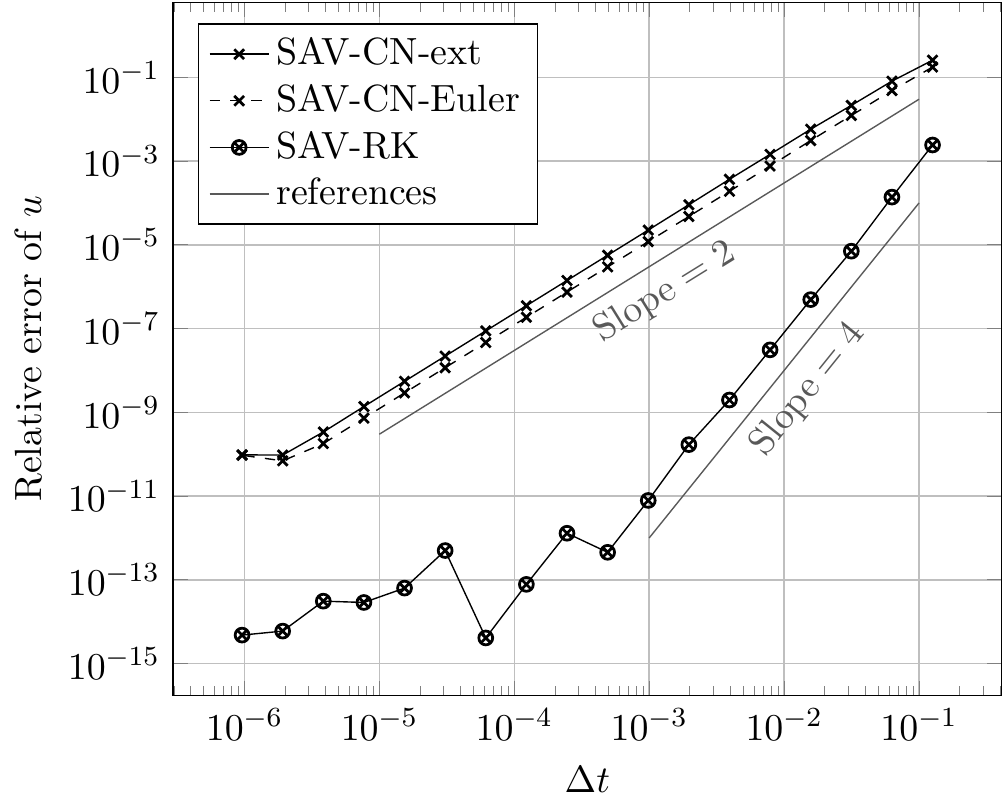}
        \caption{Relative error of the solution.}
        \label{subfig:kdv-rate-sol}
    \end{subfigure}
    \hfill
    \begin{subfigure}{0.45\linewidth}
        \centering
        \includegraphics[page=2,width=\linewidth]{figures-kdv.pdf}
        \caption{Relative error of the energy functional.}
        \label{subfig:kdv-rate-Hamil}
    \end{subfigure}
    \caption{Convergence rates of the proposed schemes for the KdV equation \eqref{eq:kdv}. The gray-colored thin lines are drawn for references}
    \label{fig:kdv-rate}
\end{figure}

\section{Concluding Remarks}
\label{sec:cr}
In this paper, we considered the SAV approach for general gradient flow \eqref{eq:abstract-gf} that has possibly lower-unbounded energy or energy functional.
We decomposed the energy functional into three parts as in \eqref{eq:assumption} and introduced two auxiliary variables.
Then, we obtained the gradient system expression of the SAV approach \eqref{eq:abstract-SAV-gf-short}, which inherits the structure of the original equation.
According to this expression, we proposed the second and fourth order SAV schemes.
In particular, the fourth order scheme is based on the canonical Runge-Kutta method and thus the novel expression \eqref{eq:abstract-SAV-gf-short} plays an essential role.
We finally presented some numerical experiments that support the theoretical results.

However, we do not mention the well-posedness of the scheme even for the second order scheme.
Indeed, it is not clear that the linear equation \eqref{eq:sav-cn-forward}, which appears in the procedure of the block Gauss elimination, is solvable.
Furthermore, for the fourth order scheme, it is not clear that the equation for the solution $u$ and the auxiliary variables $r_\rX$ ($\rX = \rL,\rU$) can be separated
unlike the second order schemes, for which the separation is possible by the block Gauss elimination.
Finally, we should discuss convergence rate of the schemes theoretically. 
We leave these studies for future work.

\section*{Acknowledgments}
The first author was supported by JSPS Grant-in-Aid for Early-Career Scientists (No.~19K14590).
The second author was supported by JSPS Grant-in-Aid for Research Activity Start-up (No.~19K23399).

\bibliographystyle{plain}
\bibliography{reference}

\end{document}